\documentclass[twoside, a4paper,12pt]{amsart}

\usepackage{tikz}
\usepackage{amsfonts}
\usepackage[hmarginratio=1:1]{geometry}

\theoremstyle{plain}
\newtheorem*{claim*}{Claim}
\newtheorem{thm}{Theorem}[section]
\newtheorem{corollary}[thm]{Corollary}
\newtheorem{lemma}[thm]{Lemma}
\newtheorem{prop}[thm]{Proposition}
\theoremstyle{definition}
\newtheorem{defn}[thm]{Definition}
\newtheorem{ex}[thm]{Example}

\newtheorem{remark}[thm]{Remark}
\newtheorem{con}[thm]{Construction}
\newtheorem{prob}[thm]{Open Problem}
\begin{document}
\title{\large{Generators and presentations for direct and wreath products of monoid acts}}
\author{Craig Miller}
\address{School of Mathematics and Statistics, St Andrews, Scotland, UK, KY16 NSS}
\email{cm380@st-andrews.ac.uk}

\begin{abstract}
We investigate the preservation of the properties of being finitely generated and finitely presented under both direct and wreath products of monoid acts.
A monoid $M$ is said to {\em preserve} property $\mathcal{P}$ in direct products if, for any two $M$-acts $A$ and $B$,
the direct product $A\times B$ has property $\mathcal{P}$ if and only if both $A$ and $B$ have property $\mathcal{P}$.
It is proved that the monoids $M$ that preserve finite generation (resp. finitely presentability) 
in direct products are precisely those for which the diagonal $M$-act $M\times M$ is finitely generated (resp. finitely presented).
We show that a wreath product $A\wr B$ is finitely generated if and only if both $A$ and $B$ are finitely generated.
It is also proved that a necessary condition for $A\wr B$ to be finitely presented is that both $A$ and $B$ are finitely presented.
Finally, we find some sufficient conditions for a wreath product to be finitely presented.
\end{abstract}

\maketitle

\section{\large{Introduction}\nopunct}
Finite generation and finite presentability are fundamental properties in the theory of monoid acts (see \cite{Kilp}).
The related notion of {\em coherency} has been intensively studied (see \cite{Gould1}, \cite{Gould2}),
and the relationhip between the monoid property of being {\em right Noetherian} and finite presentability of acts was considered in \cite{Normak}.
In this paper we continue the work initiated in \cite{Miller} of developing a systematic theory of presentations of monoid acts.
We shall consider two different product constructions for acts, namely direct products and wreath products.
For each construction, we investigate the preservation of the properties of being finitely generated and finitely presented.\par
The paper is structured as follows.
In Section 2, we collect some basic definitions and facts about generating sets and presentations.
We study direct products in sections 3 and 4.
Section 3 is concerned with diagonal acts, which are a specific type of direct product.
We consider direct products of acts in general in Section 4.
In that section we construct a general generating set and presentation for a direct product $A\times B$. 
This leads to characterisations of the monoids $M$ that have the property that, for any two $M$-acts $A$ and $B$,
the direct product $A\times B$ is finitely generated (resp. finitely presented)
if and only if both $A$ and $B$ are finitely generated (resp. finitely presented).\par 
In Section 5, we study wreath product of acts.
We characterise the wreath products $A\wr B$ that are finitely generated.
We also construct a general presentation for a wreath product $A\wr B$,
from which we deduce results pertaining to finite presentability.

\section{\large{Preliminaries}\nopunct}

Let $M$ be a monoid with identity 1. An {\em $M$-act} is a non-empty set $A$ together with a map 
$$A\times M\to A, (a, m)\mapsto am$$
such that $a(mn)=(am)n$ and $a1=a$ for all $a\in A$ and $m, n\in M.$
For instance, $M$ itself is an $M$-act via right multiplication.\par
A subset $U$ of an $M$-act $A$ is a {\em generating set} for $A$ if for any $a\in A$, there exist $u\in U, m \in M$ such that $a=um$.
We write $A=\langle U\rangle$ if $U$ is a generating set for $A$.
An $M$-act $A$ is said to be {\em finitely generated} (resp. {\em cyclic}) if it has a finite (resp. one-element) generating set.\par
A congruence $\rho$ on an $M$-act $A$ is {\em generated} by a set $X\subseteq A\times A$ if $\rho$ is the smallest congruence containing $X$,
and $\rho$ is said to be {\em finitely generated} if it has a finite generating set.\par
For other basic definitions and facts about monoid acts, we refer the reader to \cite{Kilp}.\par
Now, let $A$ be an $M$-act and let $X\subseteq A\times A$.  We introduce the notation
$$\overline{X}=X\cup\{(u, v) : (v, u)\in X\},$$
which will be used throughout the paper.
For $a, b\in A$, an $X${\em -sequence connecting} $a$ and $b$ is any sequence
$$a=p_1m_1, \; q_1m_1=p_2m_2, \; q_2m_2=p_3m_3, \; \dots, \; q_km_k=b,$$
where $(p_i, q_i)\in \overline{X}$ and $m_i\in M$ for $1\leq i\leq k$.\par
We have the following basic lemma (see \cite[Section 1.4]{Kilp} for a proof):

\begin{lemma}
Let $M$ be a monoid.  Let $A$ be an $M$-act, let $\rho$ be a congruence on $A$ generated by a set $X\subseteq A\times A$, and let $a, b\in A$.
Then $(a, b)\in\rho$ if and only if either $a=b$ or there exists an $X$-sequence connecting $a$ and $b$.
\end{lemma}

A generating set $U$ for an $M$-act $A$ is a {\em basis} of $A$ if for any $a\in A,$ there exist unique $u\in U$ and $m\in M$ such that $a=um$. 
An $M$-act $A$ is said to be {\em free} if it has a basis.  
For example, the $M$-act $M$ is free with basis $\{1\}$.\par
We have the following structure theorem for free acts.\par

\begin{thm}\cite[Theorem 1.5.13]{Kilp}.
\label{freestructure}
An $M$-act $A$ is free if and only if it is $M$-isomorphic to a disjoint union of $M$-acts all of which are $M$-isomorphic to $M$.
\end{thm}

This leads to the following explicit construction of a free act.

\begin{con}\cite[Theorem 1.5.14]{Kilp}. Let $M$ be a monoid, let $X$ be a non-empty set, and consider the set $X\times M.$  
With the operation $$(x, m)n = (x, mn)$$ for all $(x, m)\in X\times M$ and $n\in M,$ the set $X\times M$ is a free $M$-act with basis $X\times\{1\}$.
We denote this $M$-act by $F_{X, M}$, although we will usually just write $F_X$.
We will also usually write $x\cdot m$ for $(x, m)$ and $x$ for $(x, 1)$.
\end{con}

An {\em ($M$-act) presentation} is a pair $\langle X\,|\,R\rangle$, where $X$ is a non-empty set and $R\subseteq F_X\times F_X$ is a relation on the free $M$-act $F_X$.
An element $x$ of $X$ is called a {\em generator}, while an element $(u, v)$ of $R$ is called a {\em (defining) relation}, and is usually written as $u=v$.\par
An $M$-act $A$ is said to be {\em defined by the presentation} $\langle X\,|\,R\rangle$ if $A$ is $M$-isomorphic to the factor act $F_X/\rho,$ 
where $\rho$ is the congruence on $F_X$ generated by $R$.\par
Let $A$ be an $M$-act and $\theta : A\to F_X/\rho$ an $M$-isomorphism, where $\rho$ is a congruence on $F_X$ generated by $R.$
We say an element $w\in F_X$ {\em represents} an element $a\in A$ if $a\theta=[w]_{\rho}$.
For $w_1, w_2\in F_X$, we write $w_1\equiv w_2$ if $w_1$ and $w_2$ are equal in $F_X$, and $w_1=w_2$ if they represent the same element of $A$.

\begin{defn}
Let $\langle X\,|\,R\rangle$ be a presentation and let $w_1, w_2 \in F_X$. 
We say that the relation $w_1=w_2$ is a {\em consequence} of $R$ if $w_1\equiv w_2$ or there is an $R$-sequence connecting $w_1$ and $w_2$.\par
We say that $w_2$ is obtained from $w_1$ by an {\em application of a relation} from $R$ if there exists an $R$-sequence with only two distict terms connecting $w_1$ and $w_2$.
\end{defn}

\begin{defn}
Let $M$ be a monoid, let $A$ be an $M$-act, and let $\phi : X\times\{1\}\to A$ be a map.
Let $\theta : F_X\to A$ be the unique $M$-homomorphism extending $\phi,$ and let $R$ be a subset of $F_X\times F_X.$
We say that $A$ {\em satisfies} $R$ (with respect to $\phi$) if $u\theta=v\theta$ for every $(u, v)\in R.$
\end{defn}

The following fact will be used throughout the paper, usually without explicit mention.

\begin{prop}
\label{presentationcriteria}
Let $M$ be a monoid, let $A$ be an $M$-act generated by a set $X,$ let $R\subseteq F_X\times F_X,$
and let $\phi$ be the canonical map $X\times\{1\}\to X.$
Then $\langle X\,|\,R\rangle$ is a presentation for $A$ if and only if the following conditions hold:
\begin{enumerate}
 \item $A$ satisfies $R$ (with respect to $\phi$);
 \item if $w_1, w_2 \in F_X$ such that $A$ satisfies $w_1=w_2$, then $w_1=w_2$ is a consequence of $R.$
\end{enumerate}
\end{prop}

\begin{defn}
A {\em finite presentation} is a presentation $\langle X\,|\,R\rangle$ where $X$ and $R$ are finite.  
An $M$-act $A$ is {\em finitely presented} if it can be defined by a finite presentation.
\end{defn}

The property of being finitely presented is independent of the choice of a finite generating set:

\begin{prop}\cite[Proposition 3.8]{Miller}
\label{invariance}
Let $M$ be a monoid, let $A$ be an $M$-act defined by a finite presentation $\langle X\,|\,R\rangle$, and let $Y$ be another finite generating set for $A$.
Then $A$ can be defined by a finite presentation in terms of $Y.$
\end{prop}

\begin{corollary}
\label{invariancecorollary}
Let $M$ be a monoid, and let $A$ be a finitely presented $M$-act defined by a presentation $\langle X\,|\,S\rangle$ where $X$ is finite and $S$ is infinite.  
Then there exists a finite subset $S^{\prime}\subseteq S$ such that $A$ is defined by the finite presentation $\langle X\,|\,S^{\prime}\rangle$.
\end{corollary}

Let $M$ be a monoid with a generating set $X$, and let $A$ be an $M$-act.
It is clear that $A$ is defined by the presentation
$$\langle A\,|\,a\cdot x=ax\; (a\in A, x\in X)\rangle.$$
Therefore, we have:

\begin{lemma}
If $M$ is a finitely generated monoid and $A$ is a finite $M$-act, then $A$ is finitely presented.
\end{lemma}

However, it was shown in \cite[Example 3.12]{Miller} that there exist monoids $M$ for which the trivial $M$-act is not finitely presented,
and the trivial $M$-act being finitely presented is not equivalent to $M$ being finitely generated \cite[Remark 3.13]{Miller}.

\section{Diagonal acts}

\noindent For any monoid $M$, the set $M\times M$ can be made into an $M$-act by defining
$$(a, b)c=(ac, bc)$$
for all $a, b, c\in M;$ we refer to it as the {\em diagonal $M$-act}.
Diagonal acts were first mentioned, implicitly, in a problem in the American Mathematical Monthly \cite{Bulman-Fleming},
and have since been intensively studied by several authors (see \cite{Gallagher1}, \cite{Gallagher2}, \cite{Robertson1}).
A systematic study of finite generation of diagonal acts was undertaken by Gallagher in his PhD thesis \cite{Gallagher}.
He showed that infinite monoids from various `standard' monoid classes, 
such as commutative, inverse, idempotent, cancellative, completely regular and completely simple,
do not have finitely generated diagonal acts (see \cite{Gallagher1}).
However, it was shown in \cite{Gallagher2} that the diagonal act is cyclic for various transformation monoids on an infinite set.\par 
In this section we consider both finite generation and finite presentability of diagonal acts, 
primarily with the next section in mind where diagonal acts will play a key role.\par
We begin with finite generation.  
As mentioned above, it is easy to find monoids that have a non-finitely generated diagonal act.
We present the following example:

\begin{ex}
Let $M$ be an infinite monoid such that $M\!\setminus\!\{1\}$ is a subsemigroup.
We claim that $\{1\}\times M$ is contained in any generating set $U$ for $M\times M,$
and hence $M\times M$ is not finitely generated.
Indeed, if $m\in M,$ then $(1, m)=(u, v)n$ for some $(u, v)\in U$ and $n\in M,$ so $1=un$ and $m=vn.$
Since $M\!\setminus\!\{1\}$ is an ideal of $M$, we must have that $u=n=1$ and hence $m=v.$
\end{ex}

We now make the following observation.

\begin{lemma}
\label{dagen}
Let $M$ be a monoid.
The diagonal $M$-act is finitely generated if and only if it is has a generating set of the form $U\times U$ for some finite subset $U$ of $M$.
\end{lemma}

We have the following results, due to Gallagher.

\begin{thm}\cite[Theorem 4.1.5, Corollary 4.1.9]{Gallagher}
\label{Gallagher}
Let $X$ be any infinite set, and let $M$ be any of the following transformation monoids on $X$:
\begin{itemize}
\item $\mathcal{B}_X$ (the monoid of binary relations); 
\item $\mathcal{T}_X$ (the full transformation monoid);
\item $\mathcal{P}_X$ (the monoid of partial transformations); 
\item $\mathcal{F}_X$ (the monoid of full finite-to-one transformations).
\end{itemize}
Then the diagonal $M$-act is a free cyclic $M$-act (and hence finitely presented).
\end{thm}

\begin{lemma}\cite[Lemma 2.2]{Gallagher1}
\label{daidealfg}
Let $M$ be a monoid, let $N$ be a submonoid of $M,$ and suppose that $M\!\setminus\!N$ is an ideal of $M.$
If the diagonal $M$-act is generated by a set $U\times U,$ then the diagonal $N$-act is generated by the set $V\times V$ where $V=U\cap N.$
In particular, if the diagonal $M$-act is finitely generated, then the diagonal $N$-act is finitely generated.
\end{lemma}

Given a monoid $M$, we denote by $M^0$ the monoid obtained by adjoining a zero $0$ to $M$.
The following lemma provides a generating set for the diagonal $M^0$-act using a generating set for the diagonal $M$-act.

\begin{lemma}
\label{dazerofg}
Let $M$ be a monoid.  If the diagonal $M$-act is generated by a set $U\times U,$ then the diagonal $M^0$-act is generated by the set
$$Z=\bigl((U\cup\{0\})\times (U\cup\{0\})\bigr)\!\setminus\!\{(0, 0)\}.$$
\end{lemma}

\begin{corollary}\cite[Lemma 2.5]{Gallagher1}
Let $M$ be a monoid.  Then the diagonal $M$-act is finitely generated if and only if the diagonal $M^0$-act is finitely generated.
\end{corollary}

Our next result shows that the monoid property that the diagonal act is finitely generated is preserved by direct products.

\begin{prop}
\label{dadpfg}
Let $M$ and $N$ be two monoids.
Then the diagonal $(M\times N)$-act is finitely generated if and only if both the diagonal $M$-act and the diagonal $N$-act are finitely generated.
\end{prop}

\begin{proof}
($\Rightarrow$) Suppose that the diagonal $M$-act and the diagonal $N$-act are generated by finite sets $U\times U$ and $V\times V$ respectively.
We claim that the diagonal $(M\times N)$-act is generated by $(U\times V)\times(U\times V).$
Indeed, let $(m_1, n_1), (m_2, n_2)\in M\times N.$
Then $(m_1, m_2)=(u_1, u_2)m$ for some $u_1, u_2\in U$ and $m\in M,$
and $(n_1, n_2)=(v_1, v_2)n$ for some $v_1, v_2\in V$ and $n\in N.$
Therefore, we have that
$$\bigl((m_1, n_1), (m_2, n_2)\bigr)=\bigl((u_1, v_1), (u_2, v_2)\bigr)(m, n).$$
($\Leftarrow$) Suppose that the diagonal $(M\times N)$-act is generated by a set $U\times U$ where $U\subseteq M\times N$ is finite.
We claim that $M\times M$ is generated by $U^{\prime}\times U^{\prime},$ where $U^{\prime}$ is the projection of $U$ to $M.$
Indeed, let $m_1, m_2\in M.$  We choose $n_1, n_2\in N,$ so that
$$\bigl((m_1, n_1), (m_2, n_2)\bigr)=\bigl((x_1, y_1), (x_2, y_2)\bigr)(m, n)$$
for some $(x_1, y_1), (x_2, y_2)\in U$ and $(m, n)\in M\times N.$
Hence, we have that
$$(m_1, m_2)=(x_1, x_2)m\in\langle U^{\prime}\times U^{\prime}\rangle.$$
Similarly, we have that $N\times N$ is finitely generated.
\end{proof}

We now turn our attention to finite presentability.
Recall from Theorem \ref{Gallagher} that the diagonal $M$-act is finitely presented if $M$ is any of the monoids of
binary relations, full transformations, partial transformations and full finite-to-one transformations on an infinite set.\par
We show in what follows that the monoid property that the diagonal act is finitely presented is inherited by substructures and extensions in certain situations,
and is also preserved by direct products.

\begin{prop}
\label{daidealfp}
Let $M$ be a monoid, let $N$ be a submonoid of $M,$ and suppose that $M\!\setminus\!N$ is an ideal of $M$.
If the diagonal $M$-act is finitely presented, then the diagonal $N$-act is finitely presented.
\end{prop}

\begin{proof}
Let $M\times M$ be defined by a finite presentation $\langle U\times U\,|\,R\rangle$.
By Lemma \ref{daidealfg}, we have that $N\times N=\langle V\times V\rangle$ where $V=U\cap N$.
We let $Z=V\times V$ and $R^{\prime}=R\cap(F_Z\times F_Z)$, and claim that $N\times N$ is defined by the finite presentation $\langle Z\,|\,R^{\prime}\rangle.$\par
Clearly $N\times N$ satisfies the relations $R^{\prime}.$
Now let $w_1, w_2\in F_Z$ be such that $w_1=w_2$ holds in $N\times N$.
By Proposition \ref{presentationcriteria}, we just need to show that $w_1=w_2$ is a consequence of $R^{\prime}.$
Indeed, we have that $w_1=w_2$ holds in $M\times M$, 
so there exists an $R$-sequence connecting $w_1$ and $w_2$.
Since $M\!\setminus\!N$ is an ideal of $M$, every element of $M$ appearing in this sequence must in fact belong to $N,$
so $w_1=w_2$ is a consequence of $R^{\prime}$.
\end{proof}

\begin{prop}
\label{dazerofp}
Let $M$ be a monoid.  Then the diagonal $M$-act is finitely presented if and only if the diagonal $M^0$-act is finitely presented.
\end{prop}

\begin{proof}
Suppose that $M\times M$ is defined by a finite presentation $\langle U\times U\,|\,R\rangle$.
By Lemma \ref{dazerofg}, we have that $M^0\times M^0$ is generated by the set
$$Z=\bigl((U\cup\{0\})\times (U\cup\{0\})\bigr)\!\setminus\!\{0, 0\}.$$
Since $M$ is a finitely presented $M$-act and is generated by the finite set $U$, it can be defined by a finite presentation $\langle U\,|\,S\rangle$.
We define the following sets:
\begin{align*}
S_1=&\{(u, 0)\cdot m=(v, 0)\cdot n : (u\cdot m, v\cdot n)\in S\};\\
S_2=&\{(0, u)\cdot m=(0, v)\cdot n : (u\cdot m, v\cdot n)\in S\}.
\end{align*}
We shall show that $M^0\times M^0$ is defined by the finite presentation
$$P=\langle Z\,|\,R, S_1, S_2, x\cdot 0=y\cdot 0~(x, y\in Z)\rangle.$$
Let $w_1, w_2\in F_{Z, M^0}$ be such that $w_1=w_2$ holds in $M^0\times M^0$.
We need to show that $w_1=w_2$ is a consequence of the relations from $P$.
We may assume, therefore, that $w_1=w_2$ is not one of the relations $x\cdot 0=y\cdot 0~(x, y\in Z),$
so $w_1, w_2\in F_{Z, M}.$\par
If $w_1\in F_{U\times U}$, then $w_2\in F_{U\times U}$ and $w_1=w_2$ is a consequence of $R$.\par
Suppose $w_1\in F_{U\times\{0\}}$.  We must then have that $w_2\in F_{U\times\{0\}}$.
Now, $w_1\equiv(u, 0)\cdot m$ and $w_2\equiv(v, 0)\cdot n$ for some $u, v\in U$ and $m, n\in M$.
Since $u\cdot m=v\cdot n$ holds in $M$, it is a consequence of $S$;
that is, there exists an $S$-sequence connecting $u\cdot m$ and $v\cdot n$.
Therefore, we clearly have an $S_1$-sequence connecting $w_1$ and $w_2$, so $w_1=w_2$ is a consequence of $S_1$.\par
Similarly, if $w_1\in F_{\{0\}\times U}$, then $w_1=w_2$ is a consequence of $S_2$.\par 
The converse follows from Proposition \ref{daidealfp}.
\end{proof}

\begin{prop}
\label{dadpfp}
Let $M$ and $N$ be two monoids.
Then the diagonal $(M\times N)$-act is finitely presented if and only if both the diagonal $M$-act and the diagonal $N$-act are finitely presented.
\end{prop}

\begin{proof}
For both the direct implication and the converse,
we may assume that $M\times M$ and $N\times N$ are generated by finite sets $U\times U$ and $V\times V$ respectively.
As in the proof of Proposition \ref{dadpfg}, the diagonal $(M\times N)$-act is generated by the set $Z=(U\times V)\times(U\times V).$
We now define the following maps:
$$\rho_M : F_{Z, M\times N}\to F_{U\times U, M}; 
\bigl((u_1, v_1), (u_2, v_2)\bigr)\cdot(m, n)\mapsto(u_1, u_2)\cdot m;$$
$$\rho_N : F_{Z, M\times N}\to F_{V\times V, N}, 
\bigl((u_1, v_1), (u_2, v_2)\bigr)\cdot(m, n)\mapsto(v_1, v_2)\cdot n.$$
($\Rightarrow$) Clearly it is enough to show that the diagonal $M$-act $M\times M$ is finitely presented.
Since the diagonal $(M\times N)$-act is finitely presented, it can be defined by a finite presentation $\langle Z\,|\, R\rangle.$
We now define a finite set
$$R_M=\{w_1\rho_M=w_2\rho_M : (w_1, w_2)\in R\},$$
and claim that $M\times M$ is defined by the finite presentation $\langle U\times U\,|R_M\rangle.$\par
Indeed, let $w_1\equiv(u_1, u_1^{\prime})\cdot m_1\in F_{U\times U, M}$ and $w_2\equiv(u_2, u_2^{\prime})\cdot m_2\in F_{U\times U, M}$ be such that $w_1=w_2$ in $M\times M$.
Choose $v\in V,$ and let $w\equiv\bigl((u_1, v), (u_1^{\prime}, v)\bigr)\cdot(m_1, 1)$ and $w^{\prime}\equiv\bigl((u_2, v), (u_2^{\prime}, v)\bigr)\cdot(m_2, 1).$
Then $w, w^{\prime}\in F_{Z, M\times N}$ and $w=w^{\prime}$ holds in the diagonal $(M\times N)$-act.
Therefore, there exists an $R$-sequence connecting $w$ and $w^{\prime}.$
Now, applying $\rho_M$ to this $R$-sequence, we obtain an $R_M$-sequence connecting $w_1$ and $w_2,$
so $w_1=w_2$ is a consequence of $R_M.$\par
($\Leftarrow$) We have that $M\times M$ and $N\times N$ are defined by some finite presentations $\langle U\times U\,|\, R\rangle$ and $\langle V\times V\,|\,S\rangle$ respectively.
For any set $W$, we define a map
$$\sigma_W : F_W\times F_W\to F_W, (w_1, w_2)\mapsto w_1.$$
We now define the following sets:
\begin{align*}
T_1&=\{\bigl((u_1, v), (u_2, v^{\prime})\bigr)\cdot(m, n)=\bigl((u_3, v), (u_4, v^{\prime})\bigr)\cdot(m^{\prime}, n) :\\
&\hspace{2em}\bigl((u_1, u_2)\cdot m, (u_3, u_4)\cdot m^{\prime}\bigr)\in R, (v, v^{\prime})\cdot n\in\overline{S}\sigma_Y\};\\
T_2&=\{\bigl((u, v_1), (u^{\prime}, v_2)\bigr)\cdot(m, n)=\bigl((u, v_3), (u^{\prime}, v_4)\bigr)\cdot(m, n^{\prime}) :\\
&\hspace{2em}\bigl((v_1, v_2)\cdot n, (v_3, v_4)\cdot n^{\prime}\bigr)\in S, (u, u^{\prime})\cdot n\in\overline{R}\sigma_X\}.
\end{align*}
We claim that the diagonal $M$-act is defined by the finite presentation $\langle Z\,|\,T_1, T_2\rangle.$
Indeed, let $w_1, w_2\in F_{Z, M\times N}$ be such that $w_1=w_2$ in the diagonal $(M\times N)$-act.
Now, $w_1\equiv\bigl((u_1, v_1), (u_1^{\prime}, v_1^{\prime})\bigr)\cdot(s, t)$ and $w_2\equiv\bigl((u_2, v_2), (u_2^{\prime}, u_2^{\prime})\bigr)\cdot(s^{\prime}, t^{\prime})$
for some $u_1, u_1^{\prime}, u_2, u_2^{\prime}\in U,$ $v_1, v_1^{\prime}, v_2, v_2^{\prime}\in V,$ $s, s^{\prime}\in M$ and $t, t^{\prime}\in N.$\par 
Since $w_1\rho_M=w_2\rho_M$ holds in $M\times M$, there exists an $R$-sequence
$$w_1\rho_M\equiv p_1s_1, q_1s_1\equiv p_2s_2, \dots, q_ks_k\equiv w_2\rho_M,$$
where $(p_i, q_i)\in\overline{R}$ and $s_i\in M$ for $1\leq i\leq k$.\par 
Also, since $w_1\rho_N=w_2\rho_N$ holds in $N\times N$, there exists an $S$-sequence
$$w_1\rho_N\equiv p_1^{\prime}t_1, q_1^{\prime}t_1\equiv p_2^{\prime}t_2, \dots, q_k^{\prime}t_k\equiv w_2\rho_N,$$
where $(p_i, q_i)\in\overline{R}$ and $s_i\in M$ for $1\leq i\leq k$.\par
For $i\in\{1, \dots, k\}$, let $p_i\equiv(x_i, y_i)\cdot m_i$ and $q_i\equiv(x_{i+1}, y_{i+1})\cdot m_i^{\prime},$
and let $p_1^{\prime}\equiv(v_1, v_1^{\prime})\cdot n.$
Note that $m_1s_1=s$, $m_i^{\prime}s_i=m_{i+1}s_{i+1},$ and $nt_1=t.$\par
Using a relation from $T_1$, we have
\begin{equation*}
\begin{split}
\bigl((x_i, v_1), (y_i, v_1^{\prime})\bigr)\cdot(m_i, n)(s_i, t_1)&=\bigl((x_{i+1}, v_1), (y_{i+1}, v_1^{\prime})\bigr)\cdot(m_i^{\prime}, n)(s_i, t_1)\\
&\equiv\bigl((x_{i+1}, v_1), (y_{i+1}, v_1^{\prime})\bigr)\cdot(m_{i+1}, n)(s_{i+1}, t_1).
\end{split}
\end{equation*}
Therefore, through successive applications of relations from $T_1,$ we obtain
$$w_1=\bigl((u_2, v_1), (u_2^{\prime}, v_1^{\prime})\bigr)\cdot(s^{\prime}, t).$$
By a similar argument, we have that
$$w_2=\bigl((u_2, v_1), (u_2^{\prime}, v_1^{\prime})\bigr)\cdot(s^{\prime}, t)$$
is a consequence of $T_2$.
Hence, $w_1=w_2$ is a consequence of $T_1$ and $T_2.$
\end{proof}

In the following, we show that there exist monoids for which the diagonal act is finitely generated but not finitely presented.
In order to do this, we first construct a new monoid $N$ from an arbitrary monoid $M$ and $M$-act $A$, 
and then consider how finite generation (resp. finite presentability) of the diagonal $N$-act relates to finite generation (resp. finite presentability) of both the diagonal $M$-act and $A$.

\begin{con}
Let $M$ be a monoid, and let $A$ be an $M$-act disjoint from $M$ with action 
$$A\times M\to A, (a, m)\mapsto a\cdot m.$$  
We define the following multiplication on the set $M\cup A$:
$$x\circ y =
  \begin{cases} 
   xy & \text{if } x, y\in M\\
   x\cdot y & \text{if } x\in A, y\in M\\
   y       & \text{if } x\in M\cup A, y\in A.
  \end{cases}$$
With this operation the set $M\cup A$ is a monoid with identity $1_M,$ and we denote it by $\mathcal{U}(M, A).$
Note that $M$ is a submonoid and $A$ is an ideal of $\mathcal{U}(M, A).$
\end{con}

\begin{lemma}
\label{confg}
Let $M$ be a monoid, let $A$ be an $M$-act disjoint from $M,$ and let $N=\mathcal{U}(M, A)$.
Then the diagonal $N$-act is finitely generated if and only if both the diagonal $M$-act and $A$ are finitely generated.
\end{lemma}

\begin{proof}
($\Rightarrow$) Let $N\times N=\langle U\times U\rangle$ with $U$ finite.\par
Since $N\!\setminus\!M=A$ is an ideal of $N$, we have that $M\times M=\langle V\times V\rangle,$ where $V=U\cap M,$ by Lemma \ref{daidealfg}.\par
Now let $a\in A$.  We have that $(a, 1)=(u, v)n$ for some $u, v\in U$ and $n\in N,$ 
so $a=u\circ n$ and $1=v\circ n$.
Since $A$ is an ideal of $N$, we have that $n\in M,$ so $a=u\cdot n\in\langle U\cap A\rangle.$
Hence, $A$ is generated by $U\cap A$.\par
($\Leftarrow$) Let $A=\langle X\rangle$ and $M\times M=\langle U\times U\rangle$ with $X$ and $U$ finite, and assume that $1\in U$.
We claim that $N\times N$ is generated by $V\times V$ where $V=XU\cup U.$\par
Let $n_1, n_2\in N$.
If $n_1, n_2\in M$, then $(n_1, n_2)\in\langle U\times U\rangle\subseteq\langle V\times V\rangle$.\par
Assume now that $n_1\in A$, so $n_1=x_1\cdot m_1$ for some $x_1\in X$ and $m_1\in M.$
If $n_2\in M$, then $(m_1, n_2)=(u_1, u_2)m$ for some $u_1, u_2\in U$ and $m\in M,$
and hence $$(n_1, n_2)=(x_1u_1, u_2)m\in\langle XU\times U\rangle\subseteq\langle V\times V\rangle.$$
If $n_2\in A$, so $n_2=x_2\cdot m_2$ for some $x_2\in X$ and $m_2\in M.$
We have that $(m_1, m_2)=(u_1, u_2)m$ for some $u_1, u_2\in U$ and $m\in M,$
and hence $$(n_1, n_2)=(x_1u_1, x_2u_2)m\in\langle XU\times XU\rangle\subseteq\langle V\times V\rangle,$$
as required.
\end{proof}

\begin{lemma}
\label{confp}
Let $M$ be a monoid, let $A$ be an $M$-act disjoint from $M$, and let $N=\mathcal{U}(M, A)$.
If the diagonal $N$-act is finitely presented, then both the diagonal $M$-act and $A$ are finitely presented.
\end{lemma}

\begin{proof}
Since $N\!\setminus\!M=A$ is an ideal of $N,$ we have that the diagonal $M$-act is finitely presented by Proposition \ref{daidealfp}.\par
Since $N\times N$ is finitely generated, we have that $A$ is finitely generated by Lemma \ref{confg},
so let $A$ and $M\times M$ be generated by finite sets $X$ and $U\times U$ respectively, and assume that $1\in U.$
From the proof of Lemma \ref{confg}, we have that $N\times N=\langle V\times V\rangle$ where $V=XU\cup U.$
By Proposition \ref{invariance}, $N\times N$ can be defined by a finite presentation $\langle V\times V\,|\, R\rangle.$\par 
We now prove that $A$ is finitely presented.  
Let $X^{\prime}=XU$, and for each $x^{\prime}\in X^{\prime},$ 
choose $\chi(x^{\prime})\in X$ and $\alpha(x^{\prime})\in U$ such that $x^{\prime}=\chi(x^{\prime})\alpha(x^{\prime}).$
We have a well-defined $M$-homomorphism 
$$\rho : F_{X^{\prime}\times U}\to F_X, (x^{\prime}, u)\mapsto\chi(x^{\prime})\cdot\alpha(x^{\prime}).$$
Let $R_X=\{u\rho=v\rho: (u, v)\in R\}$.
We claim that $A$ is defined by the finite presentation 
$$P=\langle X\,|\,R_X, x\cdot u=y\cdot v~(x, y\in X, u, v\in U, xu=yv)\rangle.$$
Indeed, let $w_1, w_2\in F_{X, M}$ be such that $w_1=w_2$ holds in $A$.
Now $w_1\equiv x_1\cdot m$ and $w_2\equiv x_2\cdot n$ for some $x_1, x_2\in X$ and $m, n\in M$.
We have that $(m, n)=(u, v)s$ for some $u, v\in U$ and $s\in M$.\par
Let $w\equiv(x_1u, 1)\cdot s$ and $w^{\prime}\equiv(x_2v, 1)\cdot s$.
Then $w, w^{\prime}\in F_{V\times V, N}$ and $w=w^{\prime}$ holds in $N\times N$. 
Therefore, we have that $w=w^{\prime}$ is a consequence of $R$, so there exists an $R$-sequence
$$w\equiv p_1n_1, q_1n_1\equiv p_2n_2, \dots, q_kn_k\equiv w^{\prime},$$
where $(p_i, q_i)\in\overline{R}$ and $n_i\in N$ for $1\leq i\leq k$.
Now, since $A$ is an ideal of $N$, we must have that $p_i, q_i\in F_{X^{\prime}\times U, M}$ and $n_i\in M$.
Applying $\rho$ to this $R$-sequence, we obtain an $R_X$-sequence
$$\chi(x_1u)\cdot\alpha(x_1u)s\equiv(p_1\rho)m_1, (q_1\rho)m_1\equiv(p_2\rho)m_2, \dots, (q_k\rho)m_k\equiv\chi(x_2v)\cdot\alpha(x_2v)t.$$
Now, we obtain $\chi(x_1u)\cdot\alpha(x_1u)s$ from $w_1$ by an application of the relation $x_1\cdot u=\chi(x_1u)\cdot\alpha(x_1u)$, 
and likewise we obtain $\chi(x_2v)\cdot\alpha(x_2v)t$ from $w_2$.
Therefore, $w_1=w_2$ is a consequence of the relations from $P$.
\end{proof}

\begin{corollary}
\label{dafgnonfp}
There exist monoids $N$ for which the diagonal $N$-act is finitely generated but not finitely presented.
\end{corollary}

\begin{proof}
Let $M$ be the full transformation monoid on an infinite set $X$,
and consider the ideal $I=\{\alpha\in M : |\text{Im }\alpha|<\infty\}$.
Let $\rho$ be the Rees congruence on $M$ determined by the ideal $I,$ 
and let $A$ denote the cyclic $M$-act obtained by taking the quotient of the $M$-act $M$ by $\rho$ (considered as a right congruence).
Suppose that $A$ is finitely presented.  It follows from Corollary \ref{invariancecorollary} that $\rho$ is finitely generated (as a right congruence).
This implies that $I$ is generated by some finite set $U$ as a right ideal.
Now set $$n=\text{max}\{|\text{Im }\beta| : \beta\in U\},$$
and choose $\alpha\in I$ with $|\text{Im }\alpha|>n$.
We have that $\alpha=\beta\gamma$ for some $\beta\in U$ and $\gamma\in M,$ but
$$|\text{Im }\alpha|\leq|\text{Im }\beta|\leq n,$$
so we have a contradiction.  Hence, $A$ is not finitely presented.
It now follows from Theorem \ref{Gallagher}, Lemma \ref{confg} and Lemma \ref{confp} that the diagonal act of the monoid $N=\mathcal{U}(M, A)$ is finitely generated but not finitely presented.
\end{proof}

\section{\large{General direct products}\nopunct}

\noindent Let $M$ be a monoid.
For two $M$-acts $A$ and $B$, the Cartesian product $A\times B$ becomes an $M$-act by defining 
$$(a, b)m = (am, bm)$$ for all $(a, b)\in A\times B$ and $m\in M$; 
we call it the {\em direct product} of $A$ and $B$.\par
Notice that the diagonal $M$-act is the direct product of the $M$-act $M$ with itself.\par
We say that a monoid $M$ {\em preserves property $\mathcal{P}$} in direct products if it satisfies the following:
For any two $M$-acts $A$ and $B,$ the direct product $A\times B$ satisfies property $\mathcal{P}$ if and only if both $A$ and $B$ satisfy $\mathcal{P}.$\par 
In this section we shall consider the properties $\mathcal{P}$ of being finitely generated and being finitely presented.
The main aim of the section is to characterise the monoids that preserve these properties in direct products.\par
Notice that for a finite monoid $M$, the properties for $M$-acts of being finite, finitely generated and finitely presented coincide.
It follows that preservation of either finite generation or finite presentability in direct products is a finiteness condition for monoids:
\begin{lemma}
If $M$ is a finite monoid, then $M$ preserves both finite generation and finite presentability in direct products.
\end{lemma}
Having dealt with the case where $M$ is a finite monoid, we may from now on assume that $M$ is infinite.
We first consider finite generation of direct products of acts.
Since $M$-acts $A$ and $B$ are homomorphic images of the direct product $A\times B,$ we have:

\begin{lemma}
\label{dpfactorsfg}
Let $M$ be a monoid, and let $A$ and $B$ be two $M$-acts.  
If $A\times B$ is finitely generated, then both $A$ and $B$ are finitely generated.
\end{lemma}

\begin{remark}
The converse of Lemma \ref{dpfactorsfg} does not hold.
Indeed, we have already seen that there exist monoids $M$ for which the diagonal $M$-act is not finitely generated.
\end{remark}

The following result provides a generating set for the direct product of two acts, and this will be used repeatedly throughout the remainder of the section.

\begin{prop}
\label{dpgen}
Let $M$ be a monoid, and let the diagonal $M$-act be generated by a set of the form $U\times V$ where $U, V\subseteq M$.
Let $A$ and $B$ be $M$-acts generated by sets $X$ and $Y$ respectively.
Then $A\times B$ is generated by the set $Z=XU\times YV.$
\end{prop}

\begin{proof}
Let $a\in A$ and $b\in B$.  
We have that $a=xm$ for some $x\in X$ and $m\in M$, and $b=yn$ for some $y\in Y$ and $n\in M$.
Now $(m, n)=(u, v)s$ for some $(u, v)\in U\times V$ and $s\in M$.
Hence, we have that $(a, b)=(xu, yv)s\in\langle Z\rangle$.
\end{proof}

\begin{corollary}
\label{dpgencorollary}
Let $M$ be a monoid.
Then $M$ preserves finite generation in direct products if and only if the diagonal $M$-act is finitely generated.
\end{corollary}

In the following lemma we observe the close connection between the diagonal $M$-act and the direct product of two finitely generated free $M$-acts.

\begin{lemma}
Let $M$ be a monoid, and let $A$ and $B$ be free $M$-acts with finite bases.
Then $A\times B$ is isomorphic to a disjoint union of finitely many $M$-acts all of which are isomorphic to the diagonal $M$-act.
In particular, we have that $A\times B$ is finitely generated (resp. finitely presented) if and only if the diagonal $M$-act is finitely generated (resp. finitely presented).
\end{lemma}

\begin{proof}
By Theorem \ref{freestructure}, we have that $A\cong\bigcup_{i=1}^nM_i$ and $B\cong\bigcup_{i=1}^mN_i$ where the $M_i, N_i$ are disjoint $M$-acts all of which are isomorphic to the $M$-act $M$.
Therefore, we have that $$A\times B\cong\bigcup_{i=1}^n\bigcup_{j=1}^m(M_i\times N_j).$$
The second statement follows from \cite[Corollary 5.4 (resp. Corollary 5.9)]{Miller}.
\end{proof}

We now turn our attention to finite presentability.\par
Unlike for finite generation, a direct product being finitely presented does not necessarily imply that the factors are finitely presented.
For example, if we take the free monoid $M$ on some infinite set and any finitely presented $M$-act $A$, 
then of course $A\times\{0\}\cong A$ is finitely presented,
but the trivial $M$-act $\{0\}$ is not finitely presented \cite[Example 3.12]{Miller}.
Another example, where the monoid is finitely generated, was provided by Mayr and Ru{\v s}kuc in \cite{Mayr},
and we give a brief outline of it below (see \cite[Example 3.1]{Mayr} for more details).

\begin{ex}\cite[Example 3.1]{Mayr}
\label{Mayrex}
Let $G$ be the free group on two generators $x$ and $y.$
Then $G$ is the semidirect product of $A=\langle y\rangle$ and a normal subgroup $B=\langle x^a\;(a\in A)\rangle$.
Consider the following right congruences on $G$:
\begin{align*}
\alpha&=\{(u, v)\in G\times G : Au=Av\};\\
\beta&=\{(u, v)\in G\times G : Bu=Bv\}.
\end{align*}
Now, the $G$-act $G$ is isomorphic to the direct product $G/\alpha\times G/\beta.$ 
However, the factor $G/\beta$ is not finitely presented.
Indeed, if it were finitely presented, then $\beta$ would be finitely generated.
But this would imply that $B$ is finitely generated, 
which is not the case since it is a normal subgroup of a free group of infinite index.
\end{ex}

\begin{remark}
Let $G$ be the free group on some infinite set $X,$ and choose $x\in X.$ 
Set $A=\langle X\!\setminus\!\{x\}\rangle$ and $B=\langle x^a\;(a\in A)\rangle,$ 
and define right congruences $\alpha$ and $\beta$ in the same way as in Example \ref{Mayrex}.
Then $G\cong G/\alpha\times G/\beta$ is finitely presented, but neither $G/\alpha$ nor $G/\beta$ are finitely presented.
\end{remark}

\begin{prob}
Does there exist a finitely generated monoid $M$ with $M$-acts $A$ and $B$ such that $A\times B$ is finitely presented but neither $A$ nor $B$ are finitely presented.
\end{prob}

In the case that the diagonal $M$-act is finitely generated, 
it is necessary that the factors of a finitely presented direct product are also finitely presented.

\begin{prop}
\label{dpfactorsfp}
Let $M$ be a monoid such that the diagonal $M$-act is finitely generated, and let $A$ and $B$ be two $M$-acts.
If $A\times B$ is finitely presented, then both $A$ and $B$ are finitely presented.
\end{prop}

\begin{proof}
It is clearly sufficient to show that $A$ is finitely presented.\par
Let $M\times M=\langle U\times U\rangle$ for some finite subset $U$ of $M$.
We have that $A$ and $B$ are finitely generated by Lemma \ref{dpfactorsfg}, 
so let $X$ and $Y$ be finite generating sets for $A$ and $B$ respectively.
By Proposition \ref{dpgen}, $A\times B$ is generated by $Z=X^{\prime}\times Y^{\prime},$ where $X^{\prime}=XU$ and $Y^{\prime}=YU.$
Since $A\times B$ is finitely presented, it can be defined by a finite presentation $\langle Z\,|\,R\rangle$ by Proposition \ref{invariance}.\par
For each $x^{\prime}\in X^{\prime},$ choose $\chi(x^{\prime})\in X$ and $\alpha(x^{\prime})\in U$ such that $x^{\prime}=\chi(x^{\prime})\alpha(x^{\prime}).$
We have a well-defined $M$-homomorphism 
$$\rho : F_Z\to F_X, (x^{\prime}, y^{\prime})\mapsto\chi(x^{\prime})\cdot\alpha(x^{\prime}).$$
Let $R_X=\{u\rho=v\rho: (u, v)\in R\}$.
We shall show that $A$ is defined by the finite presentation 
$$P=\langle X\,|\,R_X, x\cdot u=y\cdot v~(x, y\in X, u, v\in U, xu=yv)\rangle.$$
It is clear that $A$ satisfies the relations of $P.$\par
Let $w_1, w_2\in F_X$ be such that $w_1=w_2$ holds in $A.$
Now, $w_1\equiv x_1\cdot m$ and $w_2\equiv x_2\cdot n$ for some $x_1, x_2\in X$ and $m, n\in M.$
We have that $(m, n)=(u, v)s$ for some $u, v\in U$ and $s\in M.$\par
Choose $y^{\prime}\in Y^{\prime},$ and let $w\equiv(x_1u, y^{\prime})\cdot s$ and $w^{\prime}\equiv(x_2v, y^{\prime})\cdot s.$
Then $w, w^{\prime}\in F_Z$ and $w=w^{\prime}$ holds in $A\times B.$ 
Therefore, we have that $w=w^{\prime}$ is a consequence of $R,$ so there exists an $R$-sequence
$$w\equiv p_1m_1, q_1m_1\equiv p_2m_2, \dots, q_km_k\equiv w^{\prime},$$
where $(p_i, q_i)\in\overline{R}$ and $m_i\in M$ for $1\leq i\leq k.$
Applying $\rho$ to this $R$-sequence, we obtain an $R_X$-sequence
$$\chi(x_1u)\cdot\alpha(x_1u)s\equiv(p_1\rho)m_1, (q_1\rho)m_1\equiv(p_2\rho)m_2, \dots, (q_k\rho)m_k\equiv\chi(x_2v)\cdot\alpha(x_2v)s.$$
Now, we obtain $\chi(x_1u)\cdot\alpha(x_1u)s$ from $w_1$ by an application of the relation $x_1\cdot u=\chi(x_1u)\cdot\alpha(x_1u),$ 
and likewise we obtain $\chi(x_2v)\cdot\alpha(x_2v)s$ from $w_2.$
Therefore, $w_1=w_2$ is a consequence of the relations from $P.$
\end{proof}

\begin{corollary}
Let $M$ be a monoid.  If the diagonal $M$-act is finitely generated, then the trivial $M$-act $\{0\}$ is finitely presented.
\end{corollary}

\begin{proof}
The free cyclic $M$-act $M$ is isomorphic to the direct product $\{0\}\times M,$
so it follows from Proposition \ref{dpfactorsfp} that $\{0\}$ is finitely presented.
\end{proof}

We now turn to consider when the direct product of two finitely presented acts is finitely presented.
We have already seen that direct products do not in general inherit the property of being finitely presented from their factors: 
there exist monoids $M$ for which the diagonal $M$-act is not finitely presented.
We now present a more striking example:

\begin{ex}
\textit{There exists a monoid $M$ with a finite $M$-act $A$ such that $A$ is finitely presented but the direct product $A\times A$ is not finitely presented}.\par
Let $M$ be the monoid defined by the monoid presentation
$$\langle x_i~(i\in\mathbb{N})\,|\,x_1x_1=x_1, x_2x_1=x_1, x_1x_i=x_2, x_2x_i=x_2~(i\geq 2)\rangle.$$
Notice that the set $\{x_1, x_2\}$ is a right ideal of $M.$\par
Let $A=\{a, b\},$ and define $a1=ax_1=bx_1=a$ and $b1=ax_i=bx_i=b$ for $i\geq 2.$
One can see that this makes $A$ into an $M$-act by checking that $a(mn)=(am)n$ for all $m, n\in M.$
We claim that $A$ is defined by the finite presentation
$$\langle A\,|\,a\cdot x_1=a, a\cdot x_2=b\rangle.$$
Indeed, we have $$b\cdot x_1=(a\cdot x_2)x_1\equiv a\cdot x_1=a,$$
and in a similar way we obtain $a\cdot x_i=b$ for $i>2$ and $b\cdot x_i=b$ for $i\geq 2.$\par
We have that $A\times A$ is defined by the presentation
$$\langle A\times A\,|\,u\cdot x_1=(a, a), u\cdot x_i=(b, b)~(u\in A\times A, i\geq 2)\rangle.$$
Suppose $A\times A$ is finitely presented.  Then it can be defined by a finite presentation
$$P=\langle A\times A\,|\,u\cdot x_1=(a, a), u\cdot x_i=(b, b)~ (u\in A\times A, 2\leq i\leq k)\rangle.$$
Let $i>k$ and consider the relation $(a, b)\cdot x_i=(b, b),$ which holds in $A\times A$.
Since there does not exist any relation of the form $w=(a, b)$ in $P,$ and $x_i$ cannot be written as $x_jm$ for some $j\in\{1, \dots, k\}$ and $m\in M,$ 
therefore no relation of $P$ can be applied to $(a, b)\cdot x_i,$
and so $(a, b)\cdot x_i=(b, b)$ cannot be deduced as a consequence of the relations of $P.$
Hence, we have a contradiction and $A\times A$ is not finitely presented.
\end{ex}

In the following, we build a presentation for a direct product $A\times B$ using presentations for $A$, $B$ and the diagonal act.\par 
So, let $M$ be a monoid, let $A$ and $B$ be two $M$-acts defined by presentations $\langle X\,|\,R\rangle$ and $\langle Y\,|\,S\rangle$ respectively,
and let the diagonal $M$-act $M\times M$ be defined by a presentation $\langle U\times V\,|\,P\rangle$, where $U, V\subseteq M$.\par
For each $(m, n)\in M\times M$, choose $\delta(m, n)=\bigl(\alpha(m, n), \beta(m, n)\bigr)\in U\times V$ and $\gamma(m, n)\in M$ such that 
$$(m, n)=\delta(m, n)\gamma(m, n).$$
The following observation will be crucial in the proof of Theorem \ref{dppres} below.
\begin{lemma}
\label{cruciallemma}
Let $m_1, m_2, n_1, n_2\in M,$ and let $u=\alpha(m_2, n_2)$ and $v=\beta(m_2, n_2).$  Then
$$(m_1m_2, n_1n_2)=(m_1u, n_1v)\gamma(m_2, n_2)=\delta(m_1u, n_1v)\big(\gamma(m_1u, n_1v)\gamma(m_2, n_2)\big).$$
\end{lemma}
We now define a map 
$$\rho : F_X\times F_Y\to F_Z, (x\cdot m, y\cdot n)\mapsto\big(x\alpha(m, n), y\beta(m, n)\big)\cdot\gamma(m, n).$$
Also, given any set $W$, we define a map
$$\sigma_W : F_W\times F_W\to F_W, (w_1, w_2)\mapsto w_1.$$

\begin{thm}
\label{dppres}
Let $M$ be a monoid, and let $A$ and $B$ be two $M$-acts defined by presentations $\langle X\,|\,R\rangle$ and $\langle Y\,|\,S\rangle$ respectively.
Let the diagonal $M$-act be defined by a presentation $\langle U\times V\,|\,P\rangle$, where $U, V\subseteq M$,
and let $Z=XU\times YV$.
With the remaining notation as above, we define the following sets of relations:
\begin{align*}
T_1&=\{(xu, yv)\cdot m=(xu^{\prime}, yv^{\prime})\cdot n : x\in X, y\in Y, \big((u, v)\cdot m, (u^{\prime}, v^{\prime})\cdot n\big)\in P\};\\
T_2&=\{(w_1u, wv)\rho=(w_2u, wv)\rho : (w_1, w_2)\in R, w\in\overline{S}\sigma_Y, u\in U, v\in V\};\\
T_3&=\{(wu, w_1v)\rho=(wu, w_2v)\rho : w\in\overline{R}\sigma_X, (w_1, w_2)\in S, u\in U, v\in V\}.
\end{align*}
Then $A\times B$ is defined by the presentation $\langle Z\,|\,T_1, T_2, T_3\rangle$.
\end{thm}

\begin{proof}
By Proposition \ref{dpgen}, $A\times B$ is generated by the set $Z=XU\times YV.$\par
It is clear that $A\times B$ satisfies $T_1$.
We show that $A\times B$ satisfies $T_2$; the proof for $T_3$ is similar.\par
Let $(x\cdot m, x^{\prime}\cdot m^{\prime})\in R$, $y\cdot n\in\overline{S}\sigma_Y,$ $u\in U$ and $v\in V.$
We need to show that the relation 
$$\big(x\alpha(mu, nv), y\beta(mu, nv)\big)\cdot\gamma(mu, nv)=\big(x^{\prime}\alpha(m^{\prime}u, nv), y\beta(m^{\prime}u, nv)\big)\cdot\gamma(m^{\prime}u, nv)$$
holds in $A$.  Indeed, since $x\cdot m=x^{\prime}\cdot m^{\prime}$ holds in $A$, we have that 
\begin{align*}
\big(x\alpha(mu, nv), y\beta(mu, nv)\big)\gamma(mu, nv)&=(xmu, ynv)=(x^{\prime}m^{\prime}u, ynv)\\
&=\big(x^{\prime}\alpha(m^{\prime}u, nv), y\beta(m^{\prime}u, nv)\big)\gamma(m^{\prime}u, nv).
\end{align*}
We now make the following claim:
\begin{claim*}
Let $x\in X$ and $y\in Y,$ and let $u, u^{\prime}\in U,$ $v, v^{\prime}\in V$ and $m, n\in M$ such that $(u, v)m=(u^{\prime}, v^{\prime})n$.
Then $(xu, yv)\cdot m=(xu^{\prime}, yv^{\prime})\cdot n$ is a consequence of $T_1$.
\end{claim*}
\begin{proof}
Since $(u, v)\cdot m=(u^{\prime}, v^{\prime})\cdot n$ holds in $M\times M$, it is a consequence of $P$;
that is, there exists a $P$-sequence
\begin{equation*}
\begin{split}
&(u, v)\cdot m=\bigl((u_1, v_1)\cdot m_1\bigr)n_1, \bigl((u_2, v_2)\cdot m_1^{\prime}\bigr)n_1=\bigl((u_2, v_2)\cdot m_2\bigr)n_2, \dots,\\
&\bigl((u_{k+1}, v_{k+1})\cdot m_k^{\prime}\bigr)n_k=(u^{\prime}, v^{\prime})\cdot n,
\end{split}
\end{equation*}
where $\bigl((u_i, v_i)\cdot m_i, (u_{i+1}, v_{i+1})\cdot m_i^{\prime}\bigr)\in\overline{P}$ and $n_i\in M$ for $i\in\{1, \dots, k\}.$
Therefore, we have a $T_1$-sequence
\begin{equation*}
\begin{split}
&(xu, yv)\cdot m=\bigl((xu_1, yv_1)\cdot m_1\bigr)n_1, \bigl((xu_2, yv_2)\cdot m_1^{\prime}\bigr)n_1=\bigl((xu_2, yv_2)\cdot m_2\bigr)n_2, \dots,\\
&\bigl((xu_{k+1}, yv_{k+1})\cdot m_k^{\prime}\bigr)n_k=(xu^{\prime}, yv^{\prime})\cdot n,
\end{split}
\end{equation*}
so $(xu, yv)\cdot m=(xu^{\prime}, yv^{\prime})\cdot n$ is a consequence of $T_1$.
\end{proof}
Returning to the proof of Theorem \ref{dppres}, let $w_1, w_2 \in F_Z$ be such that $w_1=w_2$ in $A\times B$.
We have that $w_1\equiv(xu, yv)\cdot m$ and $w_2\equiv(x^{\prime}u^{\prime}, y^{\prime}v^{\prime})\cdot n$
for some $x, x^{\prime}\in X,$ $y, y^{\prime}\in Y,$ $u, u^{\prime}\in U,$ $v, v^{\prime}\in V$ and $m, n\in M$.\par
Since $x\cdot um=x^{\prime}\cdot u^{\prime}n$ holds in $A$, there exists an $R$-sequence
$$x\cdot um\equiv p_1n_1, q_1n_1\equiv p_2n_2, \dots, q_kn_k\equiv x^{\prime}\cdot u^{\prime}n,$$
where $(p_i, q_i)\in\overline{R}$ and $n_i\in M$ for $1\leq i\leq k$.\par
Since $y\cdot vm=y^{\prime}\cdot u^{\prime}n$ holds in $B$, we have an $S$-sequence
$$y\cdot vm\equiv p_1^{\prime}n_1^{\prime}, q_1^{\prime}n_1^{\prime}\equiv p_2^{\prime}n_2^{\prime}, \dots, q_l^{\prime}n_l^{\prime}\equiv y^{\prime}\cdot v^{\prime}n,$$
where $(p_i^{\prime}, q_i^{\prime})\in\overline{S}$ and $n_i^{\prime}\in M$ for $1\leq i\leq l$.\par
Since $(u, v)m=\delta(um, vm)\gamma(um, vm)$, we have that 
$$w_1=(x\cdot um, y\cdot vm)\rho\equiv(p_1n_1, p_1^{\prime}n_1^{\prime})\rho$$
is a consequence of $T_1$ by the above claim.\par
Let $p_i\equiv x_i\cdot m_i$ for each $i\in\{1, \dots, k\}$, and let $p_1^{\prime}\equiv y\cdot m^{\prime}.$
Letting $u_i=\alpha(n_i, n_1^{\prime})$ and $v_i=\beta(n_i, n_1^{\prime}),$ by Lemma \ref{cruciallemma} we have
$$(m_in_i, m^{\prime}n_1^{\prime})=\delta(m_iu_i, m^{\prime}v_i)\big(\gamma(m_iu_i, m^{\prime}v_i)\gamma(n_i, n_1^{\prime})\big).$$  
We also have the following equalities:
\begin{align*}
(p_in_i, p_1^{\prime}n_1^{\prime})\rho&\equiv\big(x_i\alpha(m_in_i, m^{\prime}n_1^{\prime}), y\beta(m_in_i, m^{\prime}n_1^{\prime})\big)\gamma(m_in_i, m^{\prime}n_1^{\prime});\\
\big((p_iu_i, p_1^{\prime}v_i)\rho\big)\gamma(n_i, n_1^{\prime})&\equiv\big(x_i\alpha(m_iu_i, m^{\prime}v_i), y\beta(m_iu_i, m^{\prime}v_i)\big)\big(\gamma(m_iu_i, m^{\prime}v_i)\gamma(n_i, n_1^{\prime})\big).
\end{align*}
It now follows from the above claim that
$$(p_in_i, p_1^{\prime}n_1^{\prime})\rho=\big((p_iu_i, p_1^{\prime}v_i)\rho\big)\gamma(n_i, n_1^{\prime})$$
is a consequence of $T_1$.  Exactly the same argument proves that
$$\big((q_iu_i, p_1^{\prime}v_i)\rho\big)\gamma(n_i, n_1)=(q_in_i, p_1^{\prime}n_1^{\prime})\rho\equiv(p_{i+1}n_{i+1}, p_1^{\prime}n_1^{\prime})\rho$$
is a consequence of $T_1$.  Also, by an application of a relation from $T_2$, we have
$$\big((p_iu_i, p_1^{\prime}v_i)\rho\big)\gamma(n_i, n_1^{\prime})=\big((q_iu_i, p_1^{\prime}v_i)\rho\big)\gamma(n_i, n_1^{\prime}).$$
Therefore, it follows that $w_1=(x^{\prime}\cdot u^{\prime}n, y\cdot vm)\rho$ is a consequence of $T_1$ and $T_2$.
By a similar argument, we have that $(x^{\prime}\cdot u^{\prime}n, y\cdot vm)\rho=w_2$ is a consequence of $T_1$ and $T_3$.
Hence, $w_1=w_2$ is a consequence of $T_1$, $T_2$ and $T_3$.
\end{proof}

\begin{corollary}
\label{dpprescorollary}
Let $M$ be a monoid.
Then $M$ preserves finite presentability in direct products if and only if the diagonal $M$-act is finitely presented.
\end{corollary}

\begin{proof}
The direct implication is obvious.  The converse follows from Proposition \ref{dpfactorsfp} and Theorem \ref{dppres}.
\end{proof}

\begin{remark}
Given both Corollary \ref{dpgencorollary} and Corollary \ref{dpprescorollary},
we may now observe that Corollary \ref{dafgnonfp} is equivalent to saying that there exist monoids that preserve finite generation but not finite presentability in direct products.
\end{remark}

\section{\large{Wreath products}\nopunct}

\noindent The wreath product is an important construction in many areas of algebra (see \cite{Meldrum}).
In 1988, Knauer and Mikhalev developed a wreath product construction for monoid acts \cite{Knauer1},
which we now briefly describe.\par
Let $M$ and $N$ be two monoids, let $A$ be an $M$-act, and let $B$ be an $N$-act.
We denote by $N^A$ the set of all mappings from $A$ to $N$, and we let $c_n$ denote the map in $N^A$ which maps every element of $A$ to $n$.
By defining, for each $\theta, \phi\in N^A,$ a map $\theta\phi\in N^A$ given by $a(\theta\phi)=(a\theta)(a\phi)$ for all $a\in A,$ 
the set $N^A$ forms a monoid with identity $c_1$.\par
Now, for any $m\in M$ and $\phi\in N^A,$ we define a map $^m\!\phi\in N^A$ by $a\hspace{0.1em}^m\!\phi=(am)\phi$ for all $a\in A$.
On the set $M\times N^A$, we define 
$$(m, \theta)(n, \phi)=(mn, \theta\hspace{0.15em}^m\!\phi)$$
for all $m, n\in M$ and $\theta, \phi\in N^A$.
With this operation, the set $M\times N^A$ is a monoid with identity $(1_M, c_{1_N})$;
we denote it by $\mathcal{W}(M, N|A)$ and call it the {\em wreath product of $M$ by $N$ through $A.$}\par 
Finally, we define an action of $\mathcal{W}(M, N|A)$ on the Cartesian product $A\times B$ by
$$(a, b)(m, \theta)=\bigl(am, b(a\theta)\bigr)$$
for all $(a, b)\in A\times B$ and $(m, \theta)\in \mathcal{W}(M, N|A).$
This operation turns $A\times B$ into a $\mathcal{W}(M, N|A)$-act; we denote it by $A\wr B$ (or $A_M\wr B_N$) and call it the {\em wreath product} of (the $M$-act) $A$ by (the $N$-act) $B$.\par
Necessary and sufficient conditions for a wreath product $A\wr B$ to be regular or inverse were given in \cite{Knauer1}, 
and characterisations for both torsion free and divisible wreath products of acts were provided in \cite{Knauer2}. \par
In this section we study the behaviour of the wreath product of monoid acts with regard to finite generation and finite presentability.\par
Our first result provides necessary and sufficient conditions for the wreath products of two acts to be finitely generated.

\begin{prop}
\label{wpfg}
Let $M$ and $N$ be two monoids, let $A$ be an $M$-act, and let $B$ be an $N$-act.
Then $A\wr B$ is a finitely generated $\mathcal{W}(M, N|A)$-act if and only if $A$ is a finitely generated $M$-act and $B$ is a finitely generated $N$-act.
\end{prop}

\begin{proof}
($\Rightarrow$) Let $U$ be a finite generating set for $A\wr B$,
and let $X$ and $Y$ be the projections of $U$ to $A$ and $B$ respectively.  Clearly $X$ and $Y$ are finite.
For any $a\in A$ and $b\in B$, we have $(a, b)=(x, y)(m, \theta)$ for some $(x, y)\in U, (m, \theta)\in\mathcal{W}(M, N|A)$,
so $a=xm$ and $b=y(x\theta)$.  Hence, we have $A=\langle X\rangle$ and $B=\langle Y\rangle$.\par
($\Leftarrow$) Let $A$ and $B$ be generated by finite sets $X$ and $Y$ respectively.
We claim that $A\wr B$ is generated by the set $X\times Y.$
Indeed, let $(a, b)\in A\wr B$.  
Now $a=xm$ for some $x\in X$ and $m\in M,$ and $b=yn$ for some $y\in Y$ and $n\in N.$
Therefore, we have that $$(a, b)=(xm, yn)=(x, y)(m, c_n),$$  
as required.
\end{proof}

We now turn our attention to finite presentability, where the situation turns out to be considerably more complicated.
We begin by demonstrating that a necessary condition for a wreath product to be finitely presented is that both the factors are finitely presented.
This is perhaps quite surprising, given that in direct products finitely presentability is not necessarily inherited by factors (Example \ref{Mayrex}).

\begin{prop}
\label{wpfactorsfp}
Let $M$ and $N$ be two monoids, let $A$ be an $M$-act, and let $B$ be an $N$-act.
If $A\wr B$ is a finitely presented $\mathcal{W}(M, N|A)$-act, then $A$ is a finitely presented $M$-act and $B$ is a finitely presented $N$-act.
\end{prop}

\begin{proof}
Since $A\wr B$ is finitely generated, we have that $A=\langle X\rangle$ and $B=\langle Y\rangle$ for some finite sets $X$ and $Y$ by Proposition \ref{wpfg}.
As in the proof of Proposition \ref{wpfg}, we have that $A\wr B=\langle U\rangle$ where $U=X\times Y$.
Since $A\wr B$ is finitely presented, we have that $A\wr B$ is defined by a finite presentation $\langle U\,|\,R\rangle$ by Proposition \ref{invariance}.\par
We denote the monoid $\mathcal{W}(M, N|A)$ by $W,$ and define the maps
$$\rho_X : F_{U, W}\to F_{X, M},\; (x, y)\cdot(m, \theta)\mapsto x\cdot m;$$ 
$$\rho_Y : F_{U, W}\to F_{Y, N},\; (x, y)\cdot(m, \theta)\mapsto y\cdot x\theta.$$
Note that in the expression $y\cdot x\theta$, the element $x$ should be interpreted as an element of $A$ rather than $F_X$.\par 
Let $R_X=\{u\rho_X=v\rho_X : (u, v)\in R\}$ and $R_Y=\{u\rho_Y=v\rho_Y : (u, v)\in R\}$.
By the definition of $A\wr B$, we have that $A$ satisfies $R_X$ and $B$ satisfies $R_Y$.
We shall show that $A$ and $B$ are defined by the finite presentations $\langle X\,|\,R_X\rangle$ and $\langle Y\,|\,R_Y\rangle$ respectively.\par
Let $u, u^{\prime}\in F_{X, M}$ be such that $u=u^{\prime}$ holds in $A$, and let $v, v^{\prime}\in F_{Y, N}$ be such that $v=v^{\prime}$ holds in B.
Now $u\equiv x\cdot m$, $u^{\prime}\equiv x^{\prime}\cdot m^{\prime}$ for some $x, x^{\prime}\in X$ and $m, m^{\prime}\in M,$ 
and $v\equiv y\cdot n$, $v^{\prime}\equiv y^{\prime}\cdot n^{\prime}$ for some $y, y^{\prime}\in Y$ and $n, n^{\prime}\in N.$\par
Let $w\equiv(x, y)\cdot(m, c_{n})$ and $w^{\prime}\equiv(x^{\prime}, y^{\prime})\cdot(m^{\prime}, c_{n^{\prime}}).$
Then $w, w^{\prime}\in F_{U, W}$ and $w=w^{\prime}$ holds in $A\wr B.$ 
Therefore, we have that $w=w^{\prime}$ is a consequence of $R,$ so there exists an $R$-sequence
\begin{equation}
w\equiv p_1(n_1, \phi_1), q_1(n_1, \phi_1)\equiv p_2(n_2, \phi_2), \dots, q_k(n_k, \phi_k)\equiv w^{\prime},
\end{equation}
where $(p_i, q_i)\in\overline{R}$ and $(n_i, \phi_i)\in W$ for $1\leq i\leq k$.\par
For each $i\in\{1, \dots, k\}$, let $p_i\equiv(x_i, y_i)\cdot(m_i, \theta_i)$ and $q_i\equiv(x_{i+1}, y_{i+1})\cdot(m_i^{\prime}, \theta_i^{\prime})$.
We then have that 
\begin{align*}
p_i(n_i, \phi_i)&\equiv(x_i, y_i)\cdot(m_in_i, \theta_i\hspace{0.1em}^{m_i}\!\phi_i),\\
q_i(n_i, \phi_i)&\equiv(x_{i+1}, y_{i+1})\cdot(m_i^{\prime}n_i, \theta_i^{\prime}\hspace{0.1em}^{m_i^{\prime}}\!\phi_i).
\end{align*}
Now we have $$\bigl(p_i(n_i, \phi_i)\bigr)\rho_X\equiv x_i\cdot m_in_i\equiv(p_i\rho_X)n_i,$$ 
and similarly $\bigl(q_i(n_i, \phi_i)\bigr)\rho_X\equiv(q_i\rho_X)n_i$.
Hence, applying $\rho_X$ to (1), we obtain an $R_X$-sequence
$$u\equiv(p_1\rho_X)n_1, (q_1\rho_X)n_1\equiv(p_2\rho_X)n_2, \dots, (q_k\rho_X)n_k\equiv u^{\prime},$$
so $u=u^{\prime}$ is a consequence of $R_X$.
We also have that $$\bigl(p_i(n_i, \phi_i)\bigr)\rho_Y\equiv y_i\cdot x_i\bigl(\theta_i\hspace{0.1em}^{m_i}\!\phi_i\bigr)\equiv y_i\cdot(x_i\theta_i)(x_im_i)\phi_i\equiv(p_i\rho_Y)(x_im_i)\phi_i,$$
and similarly $\bigl(q_i(n_i, \phi_i)\bigr)\rho_Y\equiv(q_i\rho_Y)(x_{i+1}m_i^{\prime})\phi_i$.
Since $p_i=q_i$ holds in $A\wr B$, we have that $x_im_i=x_{i+1}m_i^{\prime}$.
Hence, applying $\rho_Y$ to (1), we obtain an $R_Y$-sequence
$$v\equiv(p_1\rho_Y)(x_1m_1)\phi_1, (q_1\rho_Y)(x_1m_1)\phi_1\equiv(p_2\rho_Y)(x_2m_2)\phi_2, \dots, (q_k\rho_Y)(x_km_k)\phi_k\equiv v^{\prime},$$
so $v=v^{\prime}$ is a consequence of $R_Y$.
\end{proof}

We now provide a general presentation for the wreath products of two acts.  

\begin{thm}
\label{wppres}
Let $M$ and $N$ be two monoids.  
Let $A$ be an $M$-act defined by a presentation $\langle X\,|\,R\rangle,$
and let $B$ be an $N$-act defined by a presentation $\langle Y\,|\,S\rangle.$
We define the following sets of relations:
\begin{align*}
T_1&=\{(x, y)\cdot(1, \theta)=(x, y)\cdot(1, c_{x\theta}) : x\in X, y\in Y, \theta\in N^A\};\\
T_2&=\{(x, y)\cdot(m, c_1)=(x^{\prime}, y)\cdot(m^{\prime}, c_1) : (x\cdot m, x^{\prime}\cdot m^{\prime})\in R, y\in Y\};\\
T_3&=\{(x, y)\cdot(1, c_n)=(x, y^{\prime})\cdot(1, c_{n^{\prime}}) : x\in X, (y\cdot n, y^{\prime}\cdot n^{\prime})\in S\}.
\end{align*}
Then $A\wr B$ is defined by the presentation $\langle X\times Y\,|\,T_1, T_2, T_3\rangle$.
\end{thm}

\begin{proof}
It is clear from the definition of $A\wr B$ that all the relations from $T_1$, $T_2$ and $T_3$ hold in $A\wr B$.\par 
We denote the monoid $\mathcal{W}(M, N|A)$ by $W$,
and let $w_1\equiv(x, y)\cdot(m, \theta)\in F_{X\times Y, W}$ and $w_2\equiv(x^{\prime}, y^{\prime})\cdot(m^{\prime}, \phi)\in F_{X\times Y, W}$
be such that $w_1=w_2$ in $A\wr B$.\par
Since $x\cdot m=x^{\prime}\cdot m^{\prime}$ holds in $A$, there exists an $R$-sequence
$$x\cdot m\equiv (x_1\cdot u_1)m_1, (x_2\cdot v_1)m_1\equiv (x_2\cdot u_2)m_2, \dots, (x_{k+1}\cdot v_k)m_k\equiv x^{\prime}\cdot m^{\prime},$$
where $(x_i\cdot u_i, x_{i+1}\cdot v_i)\in\overline{R}$ and $m_i\in M$ for $1\leq i\leq k$.\par 
Let $x\theta=n$ and $x^{\prime}\phi=n^{\prime}$.
Since $y\cdot n=y^{\prime}\cdot n^{\prime}$ holds in $B$, we have an $S$-sequence
$$y\cdot n\equiv(y_1\cdot s_1)n_1, (y_2\cdot t_1)n_1\equiv (y_2\cdot s_2)n_2, \dots, (y_{l+1}\cdot t_l)n_l\equiv y^{\prime}\cdot n^{\prime},$$
where $(y_i\cdot s_i, y_{i+1}\cdot t_i)\in\overline{S}$ and $n_i\in N$ for $1\leq i\leq l$.\par
We first apply a relation from $T_1$ to $w_1$:
$$w_1\equiv\big((x, y)\cdot(1, \theta)\big)(m, c_1)=\big((x, y)\cdot(1, c_n)\big)(m, c_1)\equiv\big((x_1, y)\cdot(u_1, c_1)\big)(m_1, c_n).$$
We then successively apply relations from $T_2$ to obtain
$$\big((x_{k+1}, y)\cdot(v_k, c_1)\big)(m_k, c_n)\equiv\big((x^{\prime}, y_1)\cdot(1, c_{s_1})\big)(m^{\prime}, c_{n_1}).$$
Then, through successive applications of relations from $T_3$, we attain
$$\big((x^{\prime}, y_{l+1})\cdot(1, c_{t_l})\big)(m^{\prime}, c_{n_l})\equiv\big((x^{\prime}, y^{\prime})\cdot(1, c_{n^{\prime}})\big)(m^{\prime}, c_1).$$
Finally, we acquire $w_2$ by an application of a relation from $T_1.$
Hence, we have that $w_1=w_2$ is a consequence of $T_1, T_2$ and $T_3.$
\end{proof}

In the following, we use Theorem \ref{wppres} to deduce some sufficient conditions for the wreath product of two finitely presented acts to be finitely presented.

\begin{defn}
Let $M$ and $N$ be two monoids, let $A$ be an $M$-act, let $a\in A$, and let $U$ be a subset of $N^A$.
For two maps $\theta$ and $\phi$ in $N^A$, we say that $\theta$ is $(U, a)$-{\em connected} to $\phi$ if there exists a sequence
$$\theta=\theta_1\psi_1, \phi_1\psi_1=\theta_2\psi_2, \dots, \phi_k\psi_k=\phi,$$
where each $\psi_i\in N^A$ and, for each $i\in\{1, \dots, k\},$ either $\theta_i\in U$ and $\phi_i=c_{a\theta_i},$ or $\phi_i\in U$ and $\theta_i=c_{a\phi_i}.$
\end{defn}

\begin{prop}
\label{wpfp}
Let $M$ and $N$ be two monoids.
Let $A$ be an $M$-act defined by a finite presentation $\langle X\,|\,R\rangle,$ and let $B$ be any finitely presented $N$-act.
Suppose there exists a finite set $U\subseteq N^A$ such that for every $\theta\in N^A$ and every $x\in X,$ 
either $\theta=c_{x\theta}$ or $\theta$ is $(U, x)$-connected to $c_{x\theta}.$
Then $A\wr B$ is finitely presented.
\end{prop}

\begin{proof}
Let $B$ be defined by a finite presentation $\langle Y\,|\,S\rangle.$
By Theorem \ref{wppres}, we have that $A\wr B$ is defined by the presentation $\langle X\times Y\,|\,T_1, T_2, T_3\rangle,$ where $T_1,$ $T_2,$ $T_3$ are as defined above.
Let $$T_1^{\prime}=\{(x, y)\cdot(1, \theta)=(x, y)\cdot(1, c_{x\theta}) : x\in X, y\in Y, \theta\in U\}\subseteq T_1.$$
We claim that $A\wr B$ is defined by the finite presentation $\langle X\times Y\,|\,T_1^{\prime}, T_2, T_3\rangle.$
Clearly it is enough to show that any relation from $T_1$ is a consequence of $T_1^{\prime}.$\par
Let $u=v$ be a relation from $T_1.$
Now $u\equiv(x, y)\cdot(1, \theta)$ and $v\equiv(x, y)\cdot(1, c_{x\theta})$ for some $x\in X$, $y\in Y$ and $\theta\in N^A.$
If $\theta=c_{x\theta}$, then $u\equiv v$, so suppose that $\theta\neq c_{x\theta}.$
Then there exists a sequence
$$\theta=\theta_1\psi_1, \phi_1\psi_1=\theta_2\psi_2, \dots, \phi_k\psi_k=c_{x\theta},$$
where each $\psi_i\in N^A$ and, for each $i\in\{1, \dots, k\}$, either $\theta_i\in U$ and $\phi_i=c_{x\theta_i}$, or $\phi_i\in U$ and $\theta_i=c_{x\phi_i}.$
Therefore, we have a $T_1^{\prime}$-sequence
\begin{align*}
&u\equiv((x, y)\cdot(1, \theta_1))(1, \psi_1), ((x, y)\cdot(1, \phi_1))(1, \psi_1)\equiv((x, y)\cdot(1, \theta_2))(1, \psi_2), \dots,\\&
((x, y)\cdot(1, \phi_k))(1, \psi_k)\equiv v,
\end{align*}
so $u=v$ is a consequence of $T_1^{\prime}.$
\end{proof}

\begin{corollary}
\label{wpfpcorollary}
Let $M$ and $N$ be two monoids, let $A$ be a finitely presented $M$-act, and let $B$ be a finitely presented $N$-act.
Suppose we have one of the following:
\begin{enumerate}
 \item $A$ is trivial;
 \item $N$ is trivial;
 \item $N$ contains a left zero;
 \item $A$ is finite and $N$ is a finitely generated monoid.
\end{enumerate}
Then $A\wr B$ is finitely presented.
\end{corollary}

\begin{proof}
(1) and (2). If $A$ is trivial, then $N^A=\{c_n : n\in N\},$ and if $N$ is trivial, then $N^A=\{c_1\}.$
Therefore, in either case, we have that $\theta=c_{x\theta}$ for any $x\in X, \theta\in N^A.$\par 
(3) Suppose that $N$ contains a left zero $z,$ and let $A=\langle X\rangle$ with $X$ finite.
For each $x\in X,$ we define a map $\phi_x\in N^A$ by $x\phi_x=1$ and $a\phi_x=z$ for all $a\neq x,$ and let $U=\{\phi_x : x\in X\}.$
For any $\theta\in N^A$ and $x\in X,$ we have a sequence
$$\theta=c_1\theta,\; \phi_x\theta=\phi_x c_{x\theta},\; c_1 c_{x\theta}=c_{x\theta},$$
so $\theta$ is $(U, x)$-connected to $c_{x\theta}.$\par
(4) Suppose that $N$ is generated by a finite set $X$ and that $A$ is finite.
For $a\in A$ and $x\in X$, define a map $\theta(a, x)$ by 
$$b\theta(a, x)=\begin{cases} 
   x & \text{if }b=a\\
   1 & \text{otherwise,}
  \end{cases}$$
and let $$U=\{\theta(a, x) : a\in A, x\in X\}.$$
Note that $c_{b\theta(a, x)}=c_1$ for any $b\neq a.$\par
Now let $\theta\in N^A$ and $a\in A.$
Consider $a^{\prime}\in A.$  We define a map $\phi\in N^A$ by
$$b\phi=\begin{cases} 
   a\theta & \text{if }b=a^{\prime}\\
   b\theta & \text{otherwise.}
  \end{cases}$$
(Note that $b\phi=bc_{a\theta}$ for $b\in\{a, a^{\prime}\}.$)
We show that $\theta$ is $(U, a)$-connected to $\phi$.\par
We have that $a\theta=x_1\dots x_m$ and $a^{\prime}\theta=y_1\dots y_n$ for some $x_i, y_i\in X.$
For each $i\in\{1, \dots, n\}$, let $\theta_i=\theta(a^{\prime}, y_i),$ 
and define a map $\lambda_i\in N^A$ by
$$b\lambda_i=\begin{cases} 
   y_{i+1}\dots y_n & \text{if }b=a^{\prime}\\
   b\theta & \text{otherwise.}
  \end{cases}$$
For each $i\in\{1, \dots, m\}$, let $\phi_i=\theta(a^{\prime}, x_i),$ 
and define a map $\mu_i\in N^A$ by
$$b\mu_i=\begin{cases} 
   x_{i+1}\dots x_m & \text{if }b=a^{\prime}\\
   b\theta & \text{otherwise.}
  \end{cases}$$
Also, let $\psi$ be the map in $N^A$ given by
$$b\psi=\begin{cases} 
   1 & \text{if }b=a^{\prime}\\
   b\theta & \text{otherwise.}
  \end{cases}$$
We now have a sequence
\begin{equation*}
\begin{split}
&\theta=\theta_1\lambda_1, c_1\lambda_1=\theta_2\lambda_2, \dots, c_1\lambda_{n-1}=\theta_n\psi, c_1\psi,\\ 
&\phi_m\psi=c_1\mu_{m-1}, \phi_{m-1}\mu_{m-1}=c_1\mu_{m-2}, \dots, \phi_1\mu_1=\phi.
\end{split}
\end{equation*}
Hence, $\theta$ is $(U, a)$-connected to $\phi$.
Continuing in this fashion (and recalling that $A$ is finite),  
we deduce that $\theta$ is $(U, a)$-connected to $c_{a\theta}.$\par
Therefore, in any of the cases (1), (2), (3) and (4), we have that $A\wr B$ is finitely presented by Proposition \ref{wpfp}.
\end{proof}

\begin{remark}
\label{wpremark}
The wreath product $F_{X, M}\wr F_{Y, N}$ is defined by the presentation $\langle X\times Y\,|\,T\rangle$ where $T=T_1.$
Therefore, if $X$ and $Y$ are finite, it follows from Corollary \ref{invariancecorollary} that Proposition \ref{wpfp} provides a necessary and sufficient condition for $F_{X, M}\wr F_{Y, N}$ to be finitely presented.
\end{remark}

In the final part of this section, we exhibit a couple of examples of monoids $M$ and $N$ such that the wreath product $M_M\wr N_N$ is not finitely presented.
In the first example, $M$ is potentially finite and $N$ is non-finitely generated.
In the second example, $M$ is infinite and $N$ is potentially finite.

\begin{ex}
Let $M$ be any non-trivial monoid, and let $N$ be the monoid formed by adjoining an identity to an infinite right zero semigroup $S.$\par
Suppose $M_M\wr N_N$ is finitely presented.
By Remark \ref{wpremark}, there exists a finite set $U\subseteq N^M$ such that every $\theta\neq c_{1_M\theta}$ in $N^M$ is $(U, 1_M)$-connected to $c_{1_M\theta}.$\par
Choose $m\in M\setminus\{1_M\}$ and $s\in S$ such that $s\neq m\phi$ for any $\phi\in U$ (this is possible since $S$ is infinite and $U$ is finite).
Now choose any map $\theta\in N^M$ such that $1_M\theta=1_N$ and $m\theta=s$.  We then have a sequence
$$\theta=\theta_1\psi_1, \phi_1\psi_1=\theta_2\psi_2, \dots, \phi_k\psi_k=c_{1_N},$$
where each $\psi_i\in N^M$ and, for each $i\in\{1, \dots, k\}$, either $\theta_i\in U$ and $\phi_i=c_{1_N},$ or $\phi_i\in U$ and $\theta_i=c_{1_N}.$\par
Since $s=m\theta=(m\theta_1)(m\psi_1)$ and $m\theta_1\neq s,$ we must have that $m\psi_1=s$.
It follows that $s=(m\theta_2)(m\psi_2),$ which in turn implies that $m\psi_2=s$.
Continuing in this way, we have that $m\psi_k=s.$
But then $mc_{1_N}=(m\phi_k)s=s\neq 1_N,$ which is a contradiction.
Hence, $M_M\wr N_N$ is not finitely presented.
\end{ex}

\begin{ex}
Let $M$ be any infinite monoid, and let $N$ be a non-trivial finitely generated group with finite generating set $X.$\par
Suppose that $M_M\wr N_N$ is finitely presented.
Then there exists a finite set $U\subseteq N^M$ such that every $\theta\neq c_{1\theta}$ in $N^M$ is $(U, 1)$-connected to $c_{1\theta}.$\par
We claim that the group $N^M$ is generated by the finite set $U\cup\{c_x : x\in X\}.$
However, $N^M$ is not finitely generated, so we have a contradiction and $M_M\wr N_N$ is not finitely presented.\par 
To prove the claim, let $\theta\in N^M,$ so there exists a sequence
$$\theta=\theta_1\psi_1, \phi_1\psi_1=\theta_2\psi_2, \dots, \phi_k\psi_k=c_{1\theta},$$
where each $\psi_i\in N^M$ and, for each $i\in\{1, \dots, k\},$ either $\theta_i\in U$ and $\phi_i=c_{1\theta_i},$ or $\phi_i\in U$ and $\theta_i=c_{1\phi_i}.$
We have that $\psi_i=\phi_i^{-1}\theta_{i+1}\psi_{i+1}$ for $i\in\{1, \dots, k-1\}$ and $\psi_k=\phi_k^{-1}c_{1\theta},$
so it follows that
$$\theta=\theta_1\phi_1^{-1}\dots\theta_k\phi_k^{-1}c_{1\theta}\in\langle U\cup\{c_x : x\in X\}\rangle.$$
\end{ex}

\section*{Acknowledgments}
The author would like to thank his supervisor, Professor Nik Ru{\v s}kuc, for helpful conversations which improved the organisation of the material of this paper,
and EPSRC for financial support.

\end{document}